%% file: main.tex
\documentclass[12pt, reqno]{amsart}
\usepackage{fancyhdr}
\input{preambolo_gg}

\usepackage{comment}
\usepackage{xcolor} 
\usepackage{float}
\definecolor{deepgreen}{rgb}{0.0,0.5,0.0} 
\usepackage{mathtools}
\mathtoolsset{showonlyrefs}

\renewcommand{\k}{\kappa}

\newcommand{\df}{\mathrm{d}}
\newcommand{\arbitraryconstant}{{c}}

\title[Strictly abnormal geodesics with a degeneracy point]{Strictly abnormal geodesics with a degeneracy point in the interior of their domain}

\author[Nicola Paddeu]{Nicola Paddeu$^\dagger$}
\address{$^\dagger$ University of Fribourg, Chemin du Mus\'ee~23, 1700 Fribourg, Switzerland}
\email{{nicola.paddeu@unifr.ch}}

\author[Alessando Socionovo]{Alessandro Socionovo$^\star$}
\address{$^\star$ Laboratoire Jacques-Louis Lions, Sorbonne Université, Université de Paris, CNRS, Inria, Paris, France}
\email{alessandro.socionovo@sorbonne-universite.fr}

\begin{document}

\begin{abstract}
In this article, we study abnormal curves in a family of sub-Riemannian manifolds of rank 2. We focus on abnormal curves whose lifts to the cotangent bundle annihilate, at an interior point of the domain, all Lie brackets of length up to three of vector fields tangent to the distribution. We present a method to prove that such curves are length-minimizing. Finally, we prove that strictly abnormal geodesics may cease to be locally length-minimizing after a change of the metric.
\end{abstract}

\maketitle


\section{Introduction}

Sub-Riemannian manifolds are metric spaces equipped with a distance that arises from a smooth Riemannian metric defined on a smooth distribution of the manifold (see \cite{ABB20}, \cite{libro_Enrico} or \cite{Mon02} for a detailed presentation of sub-Riemannian manifolds). One of the most difficult and studied problems in sub-Riemannian geometry 
is the regularity of \emph{geodesics}, i.e., the regularity of isometric embedding of intervals.

The Pontryagin Maximum Principle provides necessary conditions for a curve to be a geodesic, and curves that satisfy these necessary conditions are called {\em extremal curves}. There are two non-disjoint classes of extremal curves, named {\em normal} and {\em abnormal}. Normal curves are solutions of a smooth ODE that involves the metric, and they are smooth and locally geodesics. Instead,  abnormal curves do not depend on the metric and they may be neither smooth nor locally geodesics. Hence, the regularity of geodesics reduces to the regularity of non-normal geodesics. Consequently, abnormal curves have become one of the main topics of research in sub-Riemannian geometry, with a particular attention to {\em strictly abnormal} curves, namely, curves that are abnormal but not normal. 

The first example of a strictly abnormal geodesic was discovered by Montgomery in \cite{Mon94}, and it was later generalized in \cite{AS96,LS95}. All these curves are both smooth and locally geodesics regardless of the choice of the metric, and for a long time they constituted the only known examples of strictly abnormal geodesics. Recently, a new class of examples was discovered in \cite{CJMRSSS25, S25}, where the authors show the existence of strictly abnormal geodesics that lose regularity at a boundary point of their domain. All these examples appear in sub-Riemannian manifolds with 2-dimensional distribution.

Apart from these few models, the literature lacks examples of strictly abnormal geodesics, as well as methods to prove that particular abnormal curves are of minimal length.  The results obtained in \cite{AS96,LS95} are based on a non-degeneracy assumption on the distribution along the strictly abnormal curve. Namely, let $M$ be a sub-Riemannian manifold and $X_1,X_2$ be a frame for the distribution. An abnormal curve $\gamma:I {\;\color{black}\subset\R}\to M$ with lift $\lambda:I\to T^*M$ (see \cite[Section 3.4]{ABB20} for more details about lifts of extremal curves in the cotangent bundle) is \textit{non-degenerate} at $t\in I$ if 
\begin{equation}
\label{eq:non-degenerate-assumption}
    \max\{|\langle \lambda(t), [X_1,[X_1,X_2]]\rangle |, |\langle \lambda(t), [X_2,[X_1,X_2]]\rangle | \}\neq 0.
\end{equation}
Curves satisfying the condition \eqref{eq:non-degenerate-assumption} for all times are known as \emph{regular abnormal} or \emph{nice abnormal}, and they are the object of study of \cite{AS96, LS95}. Moreover, the non-degeneracy assumption \eqref{eq:non-degenerate-assumption} is satisfied by the strictly abnormal geodesics presented in \cite{CJMRSSS25} at every point in the interior of their domain, although it is not satisfied at a boundary point. Finding strictly abnormal geodesics that violate the non-degeneracy condition \eqref{eq:non-degenerate-assumption} at an interior point of their domain may be the first step towards proving the non-regularity of sub-Riemannian geodesics at an interior point, see \cite[Theorem~3.2]{AS96}.  

In this article, we present a broad class of new examples of strictly abnormal geodesics that do not satisfy the non-degeneracy assumption \eqref{eq:non-degenerate-assumption} at an interior point of their domain. As a consequence, we give a negative
answer the following natural question: does strictly abnormal geodesics remain locally geodesic for every choice of the metric? Notice that, for the nice abnormal minimizers studied in \cite{AS96, LS95}, this question has been positively answered, see \cite{LS95}. 

Let us introduce the sub-Riemannian manifolds containing the new examples of strictly abnormal geodesics we intend to present. Let $M$ be the manifold defined by
\begin{equation}
    \label{eq:manifold}
    M:=(-\frac{1}{2},\frac{1}{2})\times (-1,1)\times \R \subset\R^3,
\end{equation}
with coordinates $x=(x_1,x_2,x_3)\in M$. For $b\in \N$, we consider the vector fields
\begin{equation}
    \label{eq:horvfintro}
    X_1(x):=\partial_1 \quad \text{and} \quad X_2(x):=\partial_2 + x_1^2x_2^b\partial_3, \quad \text{for all } x\in M,
\end{equation}
and we equip $M$ with the distribution 
\begin{equation}
    \label{eq:distribution}
    \D:=\mathrm{span}\{X_1,X_2\}.
\end{equation} 
We define the curve $\g:(-1,1)\to M$ by setting
\begin{equation}
    \label{eq:gammaintro}
    \g(t):=(0,t,0), \quad \text{for all } t\in (-1,1).
\end{equation}
The curve $\g$ is abnormal in $(M,\D)$ for all $b\in \N$, see Section 2. Since
\begin{equation}
    [X_i,[X_1,X_2]](\g(0))=0, \quad \text{for all } i=1,2,
\end{equation}
the non-degeneracy assumption \eqref{eq:non-degenerate-assumption} is not satisfied at $t=0$, which is an interior point of the domain of $\g$.

For $a\in\N$ and for a function $\a\in C^\infty((-1,1)),[0,1])$, we equip the distribution $\D$ with the Riemannian metric $g$ defined by setting
\begin{equation}
    \label{eq:metricaintro}
    g_x(X_1,X_1)=1, \quad g_x(X_1,X_2)=0, \quad g_x(X_2,X_2)=1-\a(x_2)x_1x_2^a,
\end{equation}
for all $x\in M$. In this way, we have a family of sub-Riemannian manifolds $(M,\D,g)$ depending on the parameters $a,b\in \N$  and on the function $\a$. 

The main result of this article is the following.

\begin{theorem}
    \label{thm:esempio-dipendenza-metrica}
    For $a,b\in \N$ and $\a\in C^\oo((-1,1),[0,1])$, let $(M,\D,g)$ be the sub-Riemannian manifold defined by \eqref{eq:manifold}, \eqref{eq:distribution}, and \eqref{eq:metricaintro}. Let $\gamma:(-1,1)\to M$ be the curve defined by \eqref{eq:gammaintro}. If $b\leq 2a$ and $b$ is even, then there exists $\e>0$ such that for all $\a\in C^\oo((-1,1),[0,1])$ the curve $\gamma|_{[-\e,\e]}$ is a geodesic. 
\end{theorem}

The proof of Theorem~\ref{thm:esempio-dipendenza-metrica} relies on a partial generalization of the techniques developed in \cite[Section 7.1]{LS95}. In fact, we prove similar estimates as in \cite[Section 7.1]{LS95} without requiring the non-degeneracy assumption \eqref{eq:non-degenerate-assumption}. When the metric satisfies the assumptions of Theorem \ref{thm:esempio-dipendenza-metrica}, we use these estimates to prove that the curve $\g|_{[-\e,\e]}$ is locally geodesic. 

\medskip

We are interested in the case where the abnormal geodesics provided by Theorem \ref{thm:esempio-dipendenza-metrica} are 
strictly abnormal. For this reason, we provide a criterion that establish whether the curve $\g$ defined by \eqref{eq:gammaintro} is normal on a subinterval of its domain. Namely, we show that for $t_1,t_2\in(-1,1)$, with $t_1<t_2$, the curve $\gamma|_{[t_1,t_2]}$ is normal if and only if $\alpha(t)=0$ for all $t\in[t_1,t_2]$, see Corollary \ref{thm:abnormal-normal}. The idea behind this criterion is to  show that $\g$ is normal on a sub-interval of its domain if and only if its projection to a suitable Riemannian manifold is a geodesic on that sub-interval, see Proposition~\ref{prop:quando-normal}.

We then investigate whether the strictly abnormal geodesics provided by Theorem \ref{thm:esempio-dipendenza-metrica} remain locally geodesic after a change of the Riemannian metric. 

\begin{theorem}
    \label{thm:non-minimality}
    For $a,b\in \N$ and $\a\in C^\oo((-1,1),[0,1])$, let $(M,\D,g)$ be the sub-Riemannian manifold defined by \eqref{eq:manifold}, \eqref{eq:distribution}, and \eqref{eq:metricaintro}. Let $\gamma:(-1,1)\to M$ be the curve defined by \eqref{eq:gammaintro}. If $\alpha\equiv 1$ and $b>4a+4$, then the curve $\g$ is not locally geodesic. 
\end{theorem}

The proof of Theorem~\ref{thm:non-minimality} consists of the explicit construction of a curve that is shorter than $\g|_{[0,\e]}$ and that has the same initial and final points. This construction relies upon a cut and correction method, see also \cite{HL16, LM08, S25}.

Combining Theorems \ref{thm:esempio-dipendenza-metrica} and \ref{thm:non-minimality}, we give an example of a strictly abnormal geodesic which is not locally geodesic for every metric.

\begin{example}
    Fix $b=10$. Let $(M,\D)$ be as in \eqref{eq:manifold}-\eqref{eq:distribution}
    and let $\gamma$ be the curve defined by \eqref{eq:gammaintro}. For $a\in \N$ and $\alpha\in C^{\infty}((-1,1),[0,1])$, let $g$ be the Riemannian metric on $\D$ given by \eqref{eq:metricaintro}.
    By Theorem \ref{thm:esempio-dipendenza-metrica}, there exists $\e>0$ such that $\gamma|_{[-\e,\e]}$ is a strictly abnormal geodesic for the metric $g$ obtained by choosing $\alpha\equiv 1$ and $a\geq20$, while, by Theorem \ref{thm:non-minimality}, it is not locally geodesic for the metric $g$ defined by setting $\alpha\equiv 1$ and $a=1$.
\end{example}


\medskip

The paper is organized in the following way. In Section 2, we recall some basic definitions and facts in sub-Riemannian geometry and Hamiltonian mechanics, the  Pontryagin Maximum Principle and the definitions of normal and abnormal curves. We then prove Theorem~\ref{thm:abnormal-normal}. Section 3 is devoted to the study of minimality of abnormal curves in $(M,\D,g)$ and to the proof Theorem~\ref{thm:esempio-dipendenza-metrica}. In Section 4, we prove Theorem \ref{thm:non-minimality}. 

\medskip

\noindent {\bf Acknowledgements.} We thank Enrico Le Donne, Roberto Monti, Luca Nalon, and Luca Rizzi for their suggestions and the interesting discussions we had about the topics addressed in this article.

\medskip

\noindent {\bf Research funding.}
N. Paddeu was partially supported by the Swiss National Science Foundation
	(grant 200021-204501 `\emph{Regularity of sub-Riemannian geodesics and applications}').     

This project has received funding from the European Union’s Horizon 2020 research and innovation programme under the Marie Sk{\l}odowska-Curie grant agreement No 101034255. \includegraphics[width=0.65cm]{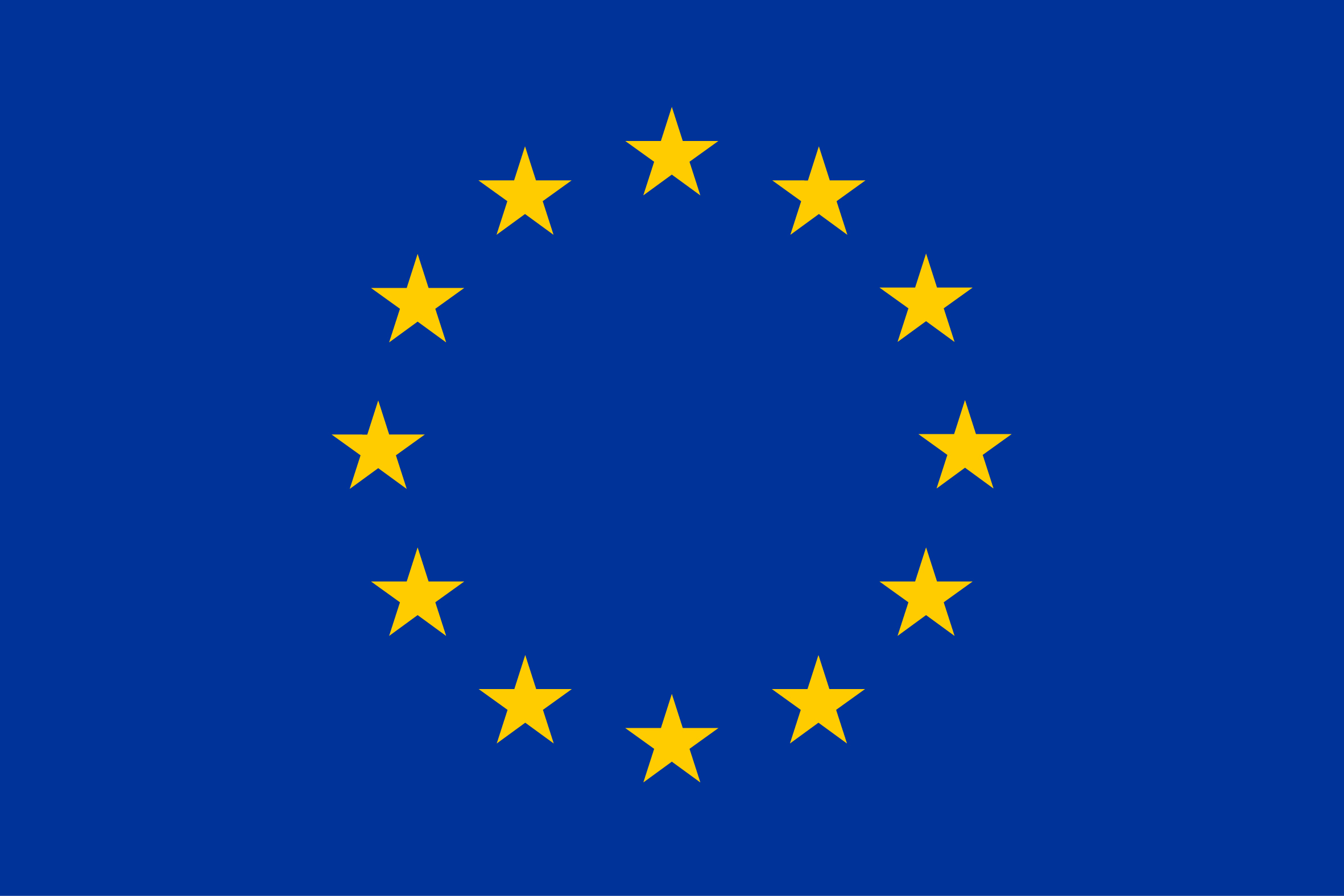}

\section{Preliminaries and proof of Theorem \ref{thm:abnormal-normal}}

Let $\Omega\subset\R^2$ be an open set and let $M:=\Omega\times\R\subset\R^3$, with coordinates $x=(x_1,x_2,x_3)\in M$. In this article, we consider  sub-Riemannian structures on $M$. Namely, we fix a function $\psi\in C^\infty(M)$ such that for all $x\in M $ there exists $n,m\in \mathbb{N}$ for which
\begin{equation}
\label{eq:b-gen-condition}
   \frac{\partial^m}{\partial x_2 ^m} \frac{\partial^n}{\partial x_1 ^n }\psi (x)\neq 0,
\end{equation}
we define the vector fields
    \begin{equation}
		\label{eq:horvf}
		X_1(x):=\partial_{x_1}, \quad 
		X_2(x):=\partial_{x_2} + \psi(x)\partial_{x_3},
	\end{equation}
and we equip $M$ with the \emph{distribution} $\D:=\vspan\{X_1,X_2\}$. Moreover, we require that the function $\psi$ does not depend on the third coordinate $x_3$, i.e., there holds
\begin{equation}
\label{eq:non-dep-psi}
    \frac{\partial}{\partial x_3}\psi(x)=0, \quad \text{for all $x\in M$}.
\end{equation}
Notice that, since $\psi$ does not depend on $x_3$, the condition~\eqref{eq:b-gen-condition} is the H\"ormander, or bracket-generating, condition. 

Let $I\subseteq \R$ be an interval. We say that an absolutely continuous curve $\eta:I \to M $ is \emph{horizontal} if $\dot \eta(t) \in \D_{\eta(t)}$ for almost every $t\in I$.
We say that a horizontal curve $\eta:I\to M$ has \emph{control} $u\in L^2(I,\mathbb{R}^2)$ if 
\begin{equation}
    \dot \eta(t)=u_1(t)X_1(\eta(t))+u_2(t)X_2(\eta(t)), \text{ for almost every } t\in I.
\end{equation}

Finally, we consider a Riemannian metric $g$ on $\D$ and the sub-Riemannian manifold $(M,\D,g)$. We require the metric $g$ to satisfy
\begin{equation}
\label{eq:invariance_of_metric}
\frac{\partial}{\partial x_3} g(X_i,X_j)=0,\quad \text{for all } i,j\in\{1,2\}.
\end{equation}

We define the \textit{sub-Riemannian length} of a horizontal curve $\eta:[0,1]\to M$ (with respect to the metric $g$) as
\begin{equation}
    L(\eta):=\int_0^1 g_{\eta(t)}(\dot\eta(t),\dot\eta(t))^{\frac12}dt,
\end{equation}
and for all $x,y\in M$, we define the
{\em distance} as 
\begin{equation}
\label{eq:def_distance}
    d(x,y):=\inf\{L(\eta) \mid \eta:[0,1]\to \R^3 \text{ horizontal}, \eta(0)=x,\eta(1)=y\}.
\end{equation}
The bracket-generating condition \eqref{eq:b-gen-condition} and Chow's Theorem \cite[Section 3.2]{ABB20} ensure that $d$ is indeed a finite value distance, and curves that realize the infimum in \eqref{eq:def_distance} are called \emph{length-minimizers}. Notice that every geodesic is a length-minimizer and every length-minimizer reparametrized by arc-length is a geodesic.

\smallskip

Necessary conditions for a horizontal curve to be a geodesic are given by the Pontryagin Maximum Principle, see Theorem \ref{thm:pmp} below. In order to state the  Pontryagin Maximum Principle, we introduce some notation and recall some basic facts about Hamiltonian mechanics.

We say that a curve $\lambda:I \to T^*M$ \textit{solves the Hamilton equations} for some function $H\in C^\oo(T^*M)$ if 
\begin{equation}
\label{eq:ham-equations}
    \dot\lambda=\Vec{H}(\lambda),
\end{equation}
where $\Vec{H}\in \Gamma(T^*M)$ solves 
\begin{equation}
\label{eq:ham-vec-field}
    \omega(\Vec{H},X)=\df H(X),\quad \text{for all } X\in \Gamma(T^*M),
\end{equation}
the $2$-form $\omega\in T^*M$ being the canonical symplectic form, see \cite[Appendix A]{Mon02}.
We refer the interested reader to \cite{ABB20, Mon02} for more details about Hamiltonian mechanics.
 The only fact of Hamiltonian mechanichs that we need in this article is the following remark, which is a direct consequence of Equations \eqref{eq:ham-equations} and \eqref{eq:ham-vec-field}.

\begin{remark}
\label{rem:hamiltoniane-differenziale-coincidente}
    Let $\lambda:I\to T^*M$ and $H,\tilde{H}\in C^\infty(T^*M)$. If $\df H_{\lambda(t)}=\df \tilde{H}_{\lambda(t)}$ for all $t\in I$, then $\lambda$ solves the Hamilton equations for $H$ if and only if it solves the Hamilton equations for $\vec{H}$.
\end{remark}

We are now ready to state the  Pontryagin Maximum Principle. We follow \cite[Theorems 3.59-4.25, Proposition 4.38]{ABB20}. 
Let $\zeta\in \Omega^1(M)$ be a $1$-form such that $\D_x =\ker(\zeta_x)$ for all $x\in M$.

\begin{theorem}[Pontryagin Maximum Principle]
\label{thm:pmp}
    Let $\eta:I\to M$ be a length-minimizing curve parametrized by multiple of arc-length. Then, at least one of the following conditions is satisfied:
\begin{itemize}
    \item[(N)] there exists a lift $\lambda:I\to T^*M $ of $\eta$ that solves the Hamilton equations for the function $H\in C^\infty(T^*M )$ defined by 
    \begin{equation}
        \label{eq:normal_Hamiltonian}
        H(\xi):=\max \{\langle\xi,X\rangle \mid X\in \Gamma(\D), \; g(X,X)=1\}, \quad \text{for all } \xi \in T^*M;
    \end{equation}

    \item[(A)]
    for all $t\in I$ there holds
    \begin{equation}
        (\zeta\wedge \df \zeta)_{\eta(t)}=0.
    \end{equation}
\end{itemize}    
\end{theorem}

Horizontal curves satisfying condition (N) of Theorem \ref{thm:pmp} are called \emph{normal}, while horizontal curves satisfying condition (A) are called \emph{abnormal}. A horizontal curve may be both normal and abnormal, and all normal curves are smooth and locally length-minimizing. Horizontal curves that are abnormal but not normal are called \emph{strictly abnormal}, and they are the subject of study of this article. 

Notice that condition (A) in Theorem \ref{thm:pmp} does not depend on the metric $g$, and it follows that a curve in $(M,\D)$ is abnormal if and only if it is supported in the set 
\begin{equation}
    \Sigma:= \{x\in\R^3 \mid (\zeta\wedge \df \zeta)_{x}=0\},
\end{equation}
which is called the \emph{Martinet surface} of $(M,\D)$.

\begin{remark}
\label{rem:our_martinet_surface}
    For a distribution $\D$ as in \eqref{eq:horvf}, the Martinet surface is given by
    \begin{equation}
\label{eq:our_martinet_surface}
    \Sigma=\Big\{x\in M \; \Big|\; \frac{\partial \psi}{\partial x_1}(x)=0\Big\}.    
    \end{equation}
Indeed, we choose $\zeta:= \df x_3-\psi \df x_2$, so that we have
\begin{equation}
    (\zeta \wedge \df \zeta)_x= -\frac{\partial \psi}{\partial x_1}(x) \df x_1\wedge \df x_2\wedge \df x_3,\quad \text{for all } x\in M.
\end{equation}
Hence, $(\zeta \wedge \df \zeta)_x=0$ if and only if $\frac{\partial \psi}{\partial x_1}(x)=0$.
    \end{remark}

\smallskip

We conclude this section stating a criterion to decide whether a horizontal curve inside the Martinet surface is normal. To do so, we study geodesics in $\Omega$ equipped with a specific Riemannian metric on $\R^2$. Namely, since the vector fields defined in~\eqref{eq:horvf} and the metric $g$ do not depend on $x_3$, see \eqref{eq:non-dep-psi} and \eqref{eq:invariance_of_metric}, we can equip $\Omega$ with the Riemannian metric $\bar{g}$ defined by
\begin{equation}
    \bar{g}(\partial_{x_i},\partial_{x_j}):=g(X_i,X_j),\quad \text{for all } i,j\in\{1,2\}.
\end{equation}
The choice of this metric makes the projection 
$\pi:M\to \Omega$, $\pi(x_1,x_2,x_3):=(x_1,x_2)$ for all $(x_1,x_2,x_3)\in M$, a submetry, see \cite{libro_Enrico}. Normal curves inside the Martinet surface are related to geodesics for the metric $\bar{g}$ by the following proposition.

 \begin{proposition}
 \label{prop:quando-normal}
     Let $\eta:I\to M$ be an abnormal curve. Assume that, for all $t\in I$, we have $\psi(\eta(t))=0$ and $\df \psi_{\eta(t)}=0$.
     Then, the curve $\eta$ is normal if and only if $\pi \circ \eta$ is a geodesic for the Riemannian metric $\bar{g}$.
 \end{proposition}

\begin{proof}
Let $v_1,v_2,w_1,w_2\in C^\infty(\Omega)$ such that $V,W\in \Gamma(\D)$ defined by
\begin{eqnarray}
    V:=v_1X_1+v_2X_2,\quad
    W:=w_1X_1+w_2X_2,
\end{eqnarray}
are an orthonormal frame for $g$.
We can write the normal Hamiltonian $H$ as

\begin{equation}
    H(\lambda)=\frac{1}{2} \left(\langle\lambda,V\rangle^2+\langle\lambda,W\rangle^2 \right), \quad \text{for all }\lambda\in T^*M,
\end{equation}
see for example \cite{Mon02}.

Define
\begin{equation}
    \tilde{H}(\lambda):= \frac{1}{2} \left(\langle\lambda,v_1\partial_{x_1}+v_2\partial_{x_2}\rangle^2+\langle\lambda,w_1\partial_{x_1}+w_2\partial_{x_2}\rangle^2 \right), \quad \text{for all }\lambda\in T^*M,
\end{equation}

Denote by $\Pi:T^*M\to M$ the canonical projection. 
We have 
\begin{eqnarray}
    H(\lambda)-\tilde{H}(\lambda)&=& \langle\lambda,v_1\partial_{x_1}+v_2\partial_{x_2}\rangle\langle \lambda ,v_2\psi(\Pi(\lambda))\partial_{x_3}\rangle+\frac{1}{2}\langle \lambda ,v_2\psi(\Pi(\lambda))\partial_{x_3}\rangle^2 \\&+&\langle\lambda,w_1\partial_{x_1}+w_2\partial_{x_2}\rangle\langle \lambda ,w_2\psi(\Pi(\lambda))\partial_{x_3}\rangle+\frac{1}{2}\langle \lambda ,w_2\psi(\Pi(\lambda))\partial_{x_3}\rangle^2,
\end{eqnarray}
Consequently, for all $\lambda \in T^*M$ for which $\psi(\Pi(\lambda))=0$ and $\df\psi_{\Pi(\lambda)}=0$, there holds
 \begin{equation}
        \label{eq:claim}
        \df H_{\lambda}-\df \tilde{H}_\lambda=0.
    \end{equation}  
In particular, for all $t\in I$ and for all $\lambda\in T_{\eta(t)}^*M$, we have
    \begin{equation}
        \label{eq:claimHam}
        \df H_{\lambda}=\df \tilde{H}_\lambda.
    \end{equation}
    
On the one hand, by Remark \ref{rem:hamiltoniane-differenziale-coincidente} and Equation \eqref{eq:claimHam}, we have that the curve $\gamma$ admits a lift to $T^*M$ solving the Hamilton equations for $H$ if and only if it admits a lift solving the Hamilton equations for $\tilde{H}$. On the other hand, it is a straightforward exercise to check that the latter condition is equivalent to $\pi\circ \gamma$ being a geodesic for $\bar{g}$ (we recall that $\pi\circ \gamma$ is a geodesic if and only if it admits a lift to $T^*\Omega$ solving the Hamilton equations for $T^*\Omega\ni \lambda \mapsto \frac{1}{2} \left(\langle\lambda,v_1\partial_{x_1}+v_2\partial_{x_2}\rangle^2+\langle\lambda,w_1\partial_{x_1}+w_2\partial_{x_2}\rangle^2 \right)$).
\end{proof}

As a consequence of Proposition \ref{prop:quando-normal}, we prove the following result.

\begin{corollary}
    \label{thm:abnormal-normal}
    For $a,b\in \N$ and $\a\in C^\oo((-1,1),[0,1])$, let $(M,\D,g)$ be the sub-Riemannian manifold defined by \eqref{eq:manifold}, \eqref{eq:distribution}, and \eqref{eq:metricaintro}. Let $\gamma:(-1,1)\to M$ be the curve defined by \eqref{eq:gammaintro}. Fix $t_1,t_2\in(-1,1)$, with $t_1<t_2$. Then, the curve $\gamma|_{[t_1,t_2]}$ is normal if and only if $\alpha(t)=0$ for all $t\in[t_1,t_2]$.
\end{corollary}

\begin{proof}[Proof of Theorem
    \ref{thm:abnormal-normal}]
Let $\Omega=(-\frac12,\frac12)\times(-1,1)$, let $X_1$ and $X_2$ be as in \eqref{eq:horvfintro}, let $\g$ be the curve defined in \eqref{eq:gammaintro}, and let $g$ be the metric given by \eqref{eq:metricaintro}. 

The curve $\g$ satisfies $\psi(\g(t))=0$ and $\df\psi_{\g(t)}=0$, for all $t\in(-1,1)$. Thus, by Proposition \ref{prop:quando-normal}, it is enough to check that $\bar{\g}:=\pi \circ \gamma|_{[t_1,t_2]}$ is a Riemannian geodesic for $\bar{g}$ if and only if for all $t\in [t_1,t_2]$ we have $\alpha(t)=0$.
Let us denote with $\Gamma$ the Christoffel symbols for the Levi-Civita connection associated to the metric $\bar{g}$. The curve $\pi\circ \bar{\g}$ is a geodesic if and only if
\begin{equation*}
    \Ddot{\bar{\gamma}}^k(t)+\dot{\bar{\g}}^i(t)\dot{\bar{\g}}^j(t)\Gamma_{ij}^k(\bar{\g})=0, \quad \text{ for all } k\in\{1,2\}, t \in [t_1,t_2],
\end{equation*}
that is, if and only if 
\begin{equation}
\label{eq:condition-riem-geodesic}
    \Gamma_{22}^k(0,t)=0 \quad \text{ for all } k\in\{1,2\}, t \in [t_1,t_2].
\end{equation}
It is an exercise to check that $\Gamma_{22}^1(0,t)=0$ for all $t\in (-1,1)$ and that 
\begin{equation*}
    \Gamma_{22}^1(0,t)=\frac{1}{2}\frac{\partial g_{22}}{\partial x_1}(0,t)=-\alpha(t)t^a,
\end{equation*}
thus the condition in \eqref{eq:condition-riem-geodesic} holds if and only $\alpha(t)=0$ for all $t\in[t_1,t_2]$.
\end{proof}

\section{Proof of Theorem \ref{thm:esempio-dipendenza-metrica}}

This section is devoted to the proof of Theorem \ref{thm:esempio-dipendenza-metrica}. We start with some preparatory considerations.

We fix $\Omega$ to be the open set $\Omega:=(-\frac12,\frac12)\times(-1,1)$, we choose $\psi(x):=x_1^2x_2^b$, with $x\in M=\Omega\times\R$, so that the vector fields $X_1$ and $X_2$ have the form in \eqref{eq:horvfintro}. We consider the metric $g$ as in \eqref{eq:metricaintro}, where $a,b\in\N$ and $\alpha\in C^\infty((-1,1),[0,1])$. Let us define the smooth function $\phi:M\to\R$,
\begin{equation}
    \phi(x):= 1 - \a(x_2)x_1^kx_2^a,
\end{equation}
so that, for all $x\in M$, we have 
\begin{equation}
    g_x(X_1,X_1)=1, \quad g_x(X_1,X_2)=g_x(X_2,X_1)=0, \quad g_x(X_2,X_2)=\phi(x).
\end{equation}

\medskip

Let $\eta:[T_1,T_2]\to M$ be a horizontal curve. A horizontal curve $\w:[0,\tau]\to M$ is a \textit{competing curve} for $\eta$ if $\w(0)=\eta(T_1)$ and $\w(\tau)=\eta(T_2)$. Notice that, if $u\in L^2([0,\tau],\R^2)$ is the control of $\w$, then we have
\begin{equation}
    \label{eq:control}
    \begin{cases}
        \dot\w_1=u_1,\\
        \dot\w_2=u_2,\\
        \dot\w_3=u_2\psi(\w).
    \end{cases}
\end{equation}
Since $\psi(x)$ does not depend on $x_3$, we can rewrite the condition of being a competing curve in terms of the first two coordinates. Indeed, let $\w:[0,\tau]\to M$ be an absolutely continuous curve with $\w(0)=\eta(T_1)$. Then $\w$ is a competing curve for $\eta$ if and only if $\w_1(\tau)=\eta_1(T_2)$, $\w_2(\tau)=\eta_2(T_2)$, and
\begin{equation}
\label{cond-end3}
    \int_0^\tau \dot\w_2\psi(\w)dt=\eta_3(T_2)-\eta_3(T_1).
\end{equation}
In order to prove Theorem \ref{thm:esempio-dipendenza-metrica}, we are going to consider either the case $\eta:=\g|_{[0,\e]}$ or the case $\eta:=\g|_{[-\e,\e]}$. In both cases, we have that $\eta_3$ is identically $0$, and thus the condition in \eqref{cond-end3} rewrites
\begin{equation}
    \label{eq:endpoint_final}
    \int_0^\tau \dot\w_2\psi(\w)dt=0.
\end{equation}

We recall that $L(\cdot)$ denotes the sub-Riemannian length of a horizontal curve.

\medskip

The proof of Theorem \ref{thm:esempio-dipendenza-metrica} consists in showing that all possible competing curves for $\bar\g_\e:=\g|_{[-\e,\e]}$ shorter than $\bar\g_\e$ are reparameterizations of $\bar\g_\e$ itself. The proof uses a technique that generalizes the one developed by Liu and Sussman in \cite[Section 7.1]{LS95}.

\begin{proof}[Proof of Theorem \ref{thm:esempio-dipendenza-metrica}.]  
    Assume that $b$ is even and $b\leq2a$.
    Fix $\e_1>0$ such that
    \begin{equation}
        \label{stima su radice}
1-\frac{1}{3}t\leq \sqrt{1-t}\leq 1-\frac{t}{2}, \quad \text{for all } t\in (-\e_1,\e_1),  
    \end{equation}
    and set 
    \begin{equation}
    \label{eq:def-epsilon}
      \e:=\min\left\{\frac{\e_1}{16},\frac18\left(2+b\right)^{-\frac{2}{b}}, \frac{1}{65}\right\}  
    \end{equation}
    
    Let $\w:[0,\tau]\to\Omega$ be a competing curve for $\bar\g_\e:=\g|_{[-\e,\e]}$ and assume that 
    \begin{equation}
        \label{eq:bycontradiction}
        L(\w)\leq L(\bar\g_\e)=2\e.
    \end{equation}
    The proof consists in estimating from above and from below the quantity 
    \[
	\displaystyle J = \int_0^\tau \psi(\w(t)) dt.
	\]
    Combining the two estimates for $J$,  we shall show that $\w$ is a reparametrization of $\gamma$.
    	
	Let $u=(u_1,u_2)$ be the control of $\w$. Since the length does not depend on the parameterization, we assume that $\w$ is parameterized by the arc-length, i.e., we have $\tau=L(\w)$ and
	\begin{align}
		\label{eq:pbal} u_1^2(t)+\phi(\w(t))u_2^2(t)= 1, \quad \text{for all } t\in[0,\tau].
	\end{align}

    We notice that, if $L_{\mathrm{eu}}$ denotes the euclidean length, then $\w$ satisfies $\frac{1}{2}L(\w)\leq L_{\mathrm{eu}}(\w)\leq 2L(\w)$. This fact, together with \eqref{eq:bycontradiction}, implies
	\begin{equation}
		\label{eq:w1w2<eps}
		\max_{t\in[0,\tau]}\{|\w_1(t)|,|\w_2(t)|\}\leq 8\e.
	\end{equation}
	
	Let us define the quantity
    \begin{equation}
        \b:= \max_{[0,\tau]}|\psi(\w)|^{\frac12}= \max_{[0,\tau]}|\w_1\w_2^{\frac b2}|,
    \end{equation}
    On the one hand, we have
	\begin{equation}
        \label{eq:J_from_above}
		\begin{split}
			J&=\int_0^\tau(1-u_2(t)\phi(\w(t))^{\frac12})\psi(\w(t))dt + \int_0^\tau u_2(t)\phi(\w(t))^{\frac12}\psi(\w(t))dt\\
            &\leq \b^2\left(\tau - \int_0^\tau u_2(t)\phi(\w(t))^{\frac12}dt\right) + \int_0^\tau u_2(t)\phi(\w(t))^{\frac12}\psi(\w(t))dt,
		\end{split}
	\end{equation}
    where in the last inequality we used 
 that $1-\phi(\w(t))u_2^2(t)\geq 0$, see \eqref{eq:pbal}. We need to estimate the two integral quantities appearing in the last line of the latter chain of inequalities.
 
    The second integral can be studied in the following way. First, by Equations \eqref{stima su radice}, \eqref{eq:def-epsilon} and \eqref{eq:w1w2<eps}, we have 
    \begin{equation}
        \label{eq:taylor_phi}
    1-\frac13\a(\w_2(t))\w_1(t)\w_2(t)^a \leq    \phi(\w(t))^{\frac12}\leq 1-\frac12\a(\w_2(t))\w_1(t)\w_2(t)^a.
    \end{equation}
    Then, using \eqref{eq:control}, \eqref{eq:endpoint_final} and \eqref{eq:taylor_phi}, we deduce
    \begin{equation}
        \begin{split}
            \int_0^\tau u_2(t)\phi(\w(t))^{\frac12}\psi(\w(t))dt&\leq -\frac12\int_0^\tau u_2(t)\psi(\w(t))\a(\w_2(t))\w_1(t)\w_2(t)^a dt\\
            &\leq  \int_0^\tau |\psi(\w(t))\w_1(t)\w_2(t)^a|dt,
        \end{split}
    \end{equation}
where in the last equality, we used \eqref{eq:pbal} and the fact that $\frac{1}{2} \leq \phi(\omega(t)) \leq \frac{3}{2}$ to bound $|u_2|\leq 2$. By assumption we have $b\leq2a$, thus we have $\w_2^a\leq \w_2^{\frac b2}$, and therefore 
    \begin{equation}
        \int_0^\tau u_2(t)\phi(\w(t))^{\frac12}\psi(\w(t))dt\leq \b^3\tau.
    \end{equation}
    Plugging the last inequality in \eqref{eq:J_from_above}, we deduce
    \begin{equation}
        \label{eq:Jabove}
        J\leq \b^2\left(\tau - \int_0^\tau u_2(t)\phi(\w(t))^{\frac12}dt\right) + \b^3\tau.
    \end{equation}
   

We next estimate $\int_0^\tau u_2\phi(\w)^{\frac12}dt$. Since $\w$ is a competing curve for $\bar\g_\e$, we have $\w_2(0)=-\e$ and $\w_2(\tau)=\e$, and thus, by \eqref{eq:control}, we get
\begin{equation}
    \label{eq:int_u2}
    \int_0^\tau u_2(t)dt= \int_0^\tau \dot\w_2(t)dt = 2\e.
\end{equation}
Using the latter equality, we deduce
	\begin{equation}
        \begin{split}
			-\int_0^\tau u_2(t)\phi(\w(t))^{\frac12}dt
            &= -\int_0^\tau u_2(t)dt -\int_0^\tau u_2(t)\left(\phi(\w(t))^{\frac12}-1\right)dt\\           
            &\stackrel{\eqref{eq:int_u2}}{=} -2\e -\int_0^\tau u_2(t)\left(\phi(\w(t))^{\frac12}-1\right)dt\\
            &\stackrel{\eqref{eq:taylor_phi}}{\leq}-2\e + \frac12 \int_0^\tau |u_2(t)||\a(\w_2(t))\w_1(t)\w_2(t)^a|dt\\
            &\leq -2\e + \int_0^\tau |\w_1(t)\w_2(t)^a|dt\\
            &\leq -2\e +  \b\tau,
        \end{split}
	\end{equation}
    where, in the last two inequalities, we used the bounds $|u_2(t)|\leq 2$, $|\alpha(\omega_2(t))|\leq 1$, and the assumption $b\leq2a$.

Hence, we can rewrite \eqref{eq:Jabove} as
\begin{equation}
    \label{eq:Jabove-meglio}
    J\leq \b^2\left(\tau - 2\e + \b \tau\right) + \b^3\tau=\beta^2(\tau-\e)+2\beta^3\tau.
\end{equation}

%
%
We next pass to the estimate for $J$ from below. We claim that
	\begin{equation}
		\label{eq:Jbelow}
		J=\int_0^\tau \psi(\w(t))dt\geq \frac{\b^3}{8}.
	\end{equation}	
 Indeed, set  $P(t)=\w_1(t)\w_2(t)^{\frac b2}$.
	By Equations
 \eqref{eq:control} and \eqref{eq:pbal} and being $\frac12\leq \phi\leq \frac32$, we have $|\dot\w_1|,|\dot\w_2|\leq 2$. Therefore, the derivative $\dot P(t) = \dot\w_1(t)\w_2(t)^{\frac b2} + q\dot\w_2(t)\w_1(t)\w_2(t)^{{\frac b2}-1}$ satisfies 
	\begin{equation}
		\label{eq:stimaderP}
		|\dot P(t)|\leq  2|\w_2(t)|^{\frac b2} + b|\w_1(t)||\w_2(t)|^{{\frac b2}-1}, \quad \text{for a.e. } t\in[0,\tau].
	\end{equation}	
	From \eqref{eq:w1w2<eps} and \eqref{eq:stimaderP}, it follows that $|\dot P(t)|\leq (2+b)8^{q} \e^{q}$ for a.e. $t\in[0,\tau]$. In particular, by \eqref{eq:def-epsilon}, we have
	\begin{equation}
		\label{eq:dotP}
		|\dot P(t)|\leq 1, \quad \text{for a.e. } t\in[0,\tau].
	\end{equation}
	Let $\bar t\in[0,\tau]$ be such that $|P(\bar t)|=\b$ and consider the interval
	\[
	\displaystyle I:=\left[\bar t- \frac{\b}{2},\bar t+ \frac{\b}{2}\right]\subset \R.
	\]
	By $P(0) = P(\tau) = 0$ and  \eqref{eq:dotP}, we have $I \subset  [0,\tau]$ and $|P(t)| \geq\beta/2$ for all $t\in I $. Since $b$ is even we have that $\psi\geq 0$, and thus we conclude 
	\begin{equation}
		\label{eq:estimate_Sussmann_below}
		J\geq\int_{I} \psi(\w(t))dt\geq \frac{\b^2}{4}|I|= \frac{\beta^3}{8}.
	\end{equation}
	The proof of claim \eqref{eq:Jbelow} is now complete. 

	Combining the two estimates \eqref{eq:Jabove-meglio} and \eqref{eq:Jbelow} for $J$, we get

	\begin{equation}
		\tau \geq 2\e + \b\left(\frac18-2\tau\right).
	\end{equation}
Since by \eqref{eq:def-epsilon} we have $4\e< \frac{1}{8}$, and $\tau=L(\omega)\leq L(\gamma)=2\e$ we have that necessarily $\b=0$. Hence, the competing curve $\w$ is a re-parameterization of $\g$.   
\end{proof}

\section{Proof of Theorem \ref{thm:non-minimality}}

We next pass to the proof of Theorem \ref{thm:non-minimality}, which consists in showing that for all $\e>0$ there exists a competing curve for $\g_\e:=\g|_{[0,\e]}$ which is of shorter length. Such a competing curve is built by using a \textquotedblleft cut and correction" method. Similar arguments to prove the non-minimality of specific horizontal curves were used in \cite{HL16, LM08, S25}.

The competing curve consists of the concatenation of the following three pieces:
\begin{itemize}
	\item[i)] for $0<\rho,h<\e$, we \textquotedblleft cut" $\g_\e$ with the curve
	\begin{equation}
		\kappa:[0,2]\to \Omega, \quad \k(t):=
		\begin{cases}
			t(h,\rho), \quad &t\in [0,1],\\
			((2-t)h,t\rho), \quad &t\in[1,2].
		\end{cases}
	\end{equation} 
	\item[ii)] we continue from the point $(0,2\rho)$ by following $\g_\e$ until its end point, i.e., we run across the curve $\g|_{[2\rho,\e]}$;
	\item[iii)] finally, we add a small rectangle at the end-point and we run across its boundary. For a suitable $0<\de<\e$, let $E_\de:=(0,\e)+[0,\de]\times [0,\e\de]$ and let $\s$ be a clockwise parameterization of $\d E_\de$.
\end{itemize}

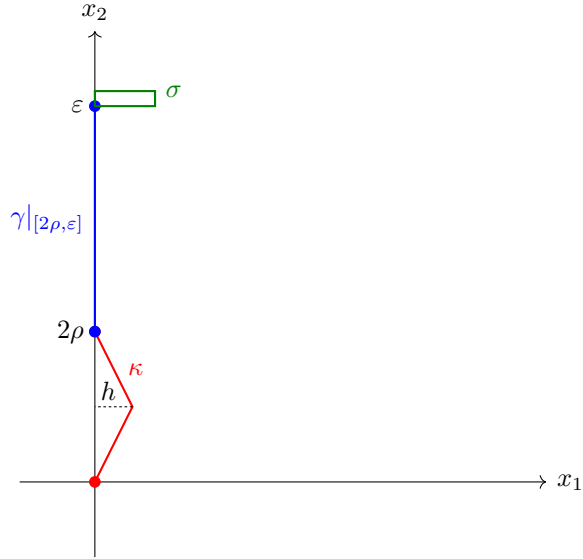
\begin{figure}[H]
    \centering
    \begin{tikzpicture}[scale=2]
        \draw[->] (-0.5,0) -- (3,0) node[right] {\footnotesize$x_1$};
        \draw[->] (0,-0.5) -- (0,3) node[above] {\footnotesize$x_2$};

        \draw[thick,red] (0,0) -- (0.25,0.5) -- (0,1);
        \filldraw[red] (0,0) circle (1pt);
        \filldraw[red] (0,1) circle (1pt);
        \node[right,red] at (0.15,0.75) {\footnotesize $\kappa$};

        \draw[thick,blue] (0,1) -- (0,2.5);
        \filldraw[blue] (0,1) circle (1pt);
        \filldraw[blue] (0,2.5) circle (1pt);
        \node[left,blue] at (0,1.75) {\footnotesize $\gamma|_{[2\rho,\e]}$};

        \draw[thick,deepgreen] (0,2.5) -- (0.4,2.5) -- (0.4,2.6) -- (0,2.6) -- cycle;
        \node[right,deepgreen] at (0.4,2.6) {\footnotesize $\sigma$};

        \draw[dashed, dash pattern=on 1pt off 1pt] (0.25,0.5) -- (0,0.5);
        
        \node[above] at (0.09,0.47) {\footnotesize $h$};

        \node[left] at (0,1) {\footnotesize $2\rho$};
        \node[left] at (0,2.5) {\footnotesize $\e$};

    \end{tikzpicture}
    \caption{The curve $\pi\circ\omega:=\kappa*\gamma|_{[2\rho, \e]}*\sigma$.}
    \label{fig:omega}
\end{figure}

We denote with $\omega$ the unique horizontal curve satisfying $\pi\circ\w=\kappa*\g|_{[\rho,\e]}*\s$, where the operation \textquotedblleft $*$" denotes the concatenation of plane curves. Notice that, by the choice of the metric $g$, the curve $\k$ is indeed shorter than $\gamma|_{[0,2\rho]}$.
For some $\rho,h \in(0,\e)$ and $\de>0$, the curve $\w$ might not be a competing curve for $\g_\e$.
The proof of Theorem \ref{thm:non-minimality}, consists in showing that, for all $\e>0$, it is possible to choose $\rho,h,\de$ in a way that $\w$ is a competing curve for $\gamma_\e$ and $L(\w)<L(\g_\e)$.

\begin{lemma}
	\label{lem:correction}
	 There exists $C,\xi>0$ such that, for every $\e>0$, and for every $\rho,h\in(0,\min\{\xi,\e\})$, there exists a unique $\delta>0$ such that the curve $\w$ is a competing curve for $\gamma_\e$. Moreover, for such $\delta$, there holds
	\begin{equation}
		\label{eq:correction}
		\de^{3}  \leq \frac{C}{\e^{b+1}} h^{2}\rho^{b+1}.
	\end{equation}
\end{lemma}

\begin{proof}
	We have to compute and compare the effects on the third coordinate made by the cut and by the rectangle.
	
	For the cut we have
	\begin{equation}
		\label{eq:errk}
		\begin{split}
			\De^{\cut}_3(\rho)&:=\int_{0}^2 \k_1(t)^{2}\k_2(t)^b\dot\k_2(t)dt\\
			&=\int_{0}^{1} (th)^{2}(t\rho)^{b}\rho dt + \int_1^2 [h(2-t)]^{2}(t\rho)^{b}\rho dt\\
			&=C_1h^{2}\rho^{b+1},
		\end{split}
	\end{equation}
    where 
    \begin{equation}
        C_1:=\int_0^1t^{2+b} \df t+\int_1^2(2-t)^{2}t^b \df t.
    \end{equation}

	For the rectangle we use the Stokes theorem. We compute
	\begin{align}
			\De_3^{\cor}(\de)&:= \int_{\sigma} x_1^{2}x_2^b \df x_2\\
            &=-\iint_{E_\de} 2x_1x_2^b dx_1dx_2\\
			&=-\int_{\e}^{\e(1+\de)}\left(\int_0^{\de}2x_1x_2^bdx_1 \right)dx_2 \qquad (x_2=\e y, \; dx_2=\e dy)\\
			&=-\frac{\e^{b+1}}{b+1}\de^2((1+\delta)^{b+1}-1)\\
            &=-\frac{\e^{b+1}}{b+1}\de^3f(\delta), \label{eq:errQ}
	\end{align}
	where $f(\delta):=\frac{(1+\delta)^{b+1}-1}{\delta}$. The function $f$ is continuous, strictly increasing on $(0,+\infty)$, and $\lim_{\delta\to 0} f(\delta)=1$, thus there exists $\delta_1>0$ such that, for all $\delta\in (0,\delta_1)$, there holds 
    \begin{equation}
        \label{eq:costante1}
        0\leq f(\delta)-1\leq \frac{1}{2}.
    \end{equation}
    Moreover, letting $\xi_1:=f(\de_1)\de_1^3>0$, we have that $f(\delta)\delta^{3}\leq \xi_1$ if and only if $\delta\in (0,\delta_1]$. Finally, we can choose $\xi>0$ such that $\frac{(b+1)C_1}{\e^{b+1}}h^{2}\rho^{b+1}\leq \xi_1$ for every $h,\rho\leq \xi$.
    
The curve $\w$ is a competing curve for $\gamma_\e$ if and only if the third coordinate of $\w$ at the final point vanishes, that is, if and only if 
\begin{equation}
    \De^{\cut}_3(\rho)+\De_3^{\cor}(\de)=0.
\end{equation}
By~\eqref{eq:errk} and~\eqref{eq:errQ} the latter condition rewrites
\begin{equation}
\label{eq:condition-on-delta}
    \frac{C_1(b+1)}{\e^{b+1}}h^{2}\rho^{b+1}=\de^{3}f(\delta).
\end{equation}
Since the right-hand side is monotone in $\delta$, it vanishes at $\delta=0$ and tends to infinity as $\delta\to \infty$, there exists a unique $\delta$ for which
Equation \eqref{eq:condition-on-delta} holds. Moreover, by \eqref{eq:costante1}, \eqref{eq:condition-on-delta}, and by the choiche of $\xi$, we get that for such $\de$ the Inequality \eqref{eq:correction} holds with $C:=C_1(b+1)$.
\end{proof}

We are now ready to finish the proof of Theorem~\ref{thm:non-minimality}.

\begin{proof}[Proof of Theorem \ref{thm:non-minimality}]
Let $C,\xi>0$ be the constants for which Lemma \ref{lem:correction} holds. Without loss of generality, we may assume $\xi<1$. Fix $\e\in (0,1)$ and $\arbitraryconstant>0$ such that $4a+4+4\arbitraryconstant\leq b$ (e.g., $c=\frac18$). Moreover, fix $\rho\in(0,\min\{\xi,\e\})$ and $h:=\rho^{a+2+\arbitraryconstant}$, with $\rho$ satisfying
\begin{equation}
    \label{eq:rho-small}
    \rho^{\arbitraryconstant}< \min\left\{\frac14, \frac{1}{32}\left(\frac{\e^{b+1}}{C}\right)^{\frac{1}{3}}\right\}.
\end{equation}

By Lemma \ref{lem:correction} there exists $\delta>0$ satisfying \eqref{eq:correction} such that $\w$ is a competing curve for $\gamma_\e$. We claim that $\w$ is shorter than $\gamma_\e$, i.e., we claim that
\begin{equation}
    \label{eq:claim_on_length}
    L(\k)+L(\g|_{[2\rho,\e]})+L(\sigma)<\e.
\end{equation}

We start by estimating the length of $\kappa$:
	\begin{equation}
		\label{eq:lenk}
        \begin{split}
		L(\k)&=\int_0^1 \sqrt{h^2+(1-(th)(t\rho)^a)\rho^2}dt + \int_1^2 \sqrt{h^2 + (1-[(2-t)h](t\rho)^a)\rho^2}dt\\
        &\leq \int_0^1 \sqrt{h^2+\rho^2}dt + \int_1^{\frac{3}{2}} \sqrt{h^2 + \left(1-\frac{3}{4}h\rho^a\right)\rho^2}dt+ \int_{\frac{3}{2}}^2 \sqrt{h^2+\rho^2}dt 
        \\
        &=\rho\left(\frac{3}{2}\sqrt{1+\left(h\rho^\1\right)^2} + \frac{1}{2}\sqrt{1+\left(h\rho^\1\right)^2-\frac{3}{4}h\rho^a}\right)\\
        &\leq \rho\Big(2+\frac{1}{4}\left(\left(h\rho^\1\right)^2-\frac34 h\rho^a\right)\Big),
        \end{split}
	\end{equation}
    where in the last inequality we used that $\sqrt{1+|x|}\leq 1+\frac{|x|}{2}$, for all $x\in\R$. By \eqref{eq:rho-small} and $h=\rho^{a+2+\arbitraryconstant}$ we get $(h\rho^\1)^2\leq \frac14 h\rho^a$, and thus we deduce 
    \begin{equation}
    \label{eq:stima_L_k}
        L(\k)\leq 2\rho - \frac18 h\rho^{a+1}.
    \end{equation}    
    
The other two lengths in \eqref{eq:claim_on_length} are easier to estimate. Since $\dot\g(t)=(0,1)$, we have
	\begin{equation}
		\label{eq:lengamma}
L(\g|_{[2\rho,\e]})=\e-2\rho,
	\end{equation}    
while the length of the rectangle $\sigma$ is estimated by 
\begin{equation}
\label{eq:square}
   L(\s)\leq 4\de. 
\end{equation}

Combining Equations \eqref{eq:stima_L_k}, \eqref{eq:lengamma}, and \eqref{eq:square}, we get
\begin{eqnarray}  \label{eq:first_estimates_total_length}
L(\k)+L(\g|_{[2\rho,\e]})+L(\sigma)&\leq& \e + 4\delta-\frac{1}{8}h\rho^{a+1}\\
&\stackrel{\eqref{eq:correction}}{\leq}& \e + 4\left(\frac{C}{\e^{b+1}}\right)^{\frac{1}{3}}h^{\frac{2}{3}}\rho^{\frac{b+1}{3}}  -\frac{1}{8}h\rho^{a+1}.
\end{eqnarray}
Since $h=\rho^{a+2+\arbitraryconstant}$, and since $b\geq 4a+4+4\arbitraryconstant$, we get
\begin{eqnarray}
L(\k)+L(\g|_{[2\rho,\e]})+L(\sigma)&\leq&  \e + 4\left(\frac{C}{\e^{b+1}}\right)^{\frac{1}{3}}\rho^{\frac{2a+4+2\arbitraryconstant}{3}}\rho^{\frac{4a+5+4\arbitraryconstant}{3}}  -\frac{1}{8}\rho^{2a+3+\arbitraryconstant}\\
&=&\e+\rho^{2a+3+\arbitraryconstant}\left(4\left(\frac{C}{\e^{b+1}}\right)^{\frac{1}{3}}\rho^\arbitraryconstant-\frac{1}{8}\right)\stackrel{\eqref{eq:rho-small}}{<}\e.
\end{eqnarray}
The proof of the claim \eqref{eq:claim_on_length}, and then of Theorem \ref{thm:non-minimality}, is concluded.
\end{proof}

\bibliographystyle{plain}
\bibliography{Biblio}

\end{document}

%% file: preambolo_gg.tex
\usepackage[utf8]{inputenc}
\usepackage[english]{babel}

\usepackage{mathrsfs}
\usepackage{amsthm}
\usepackage{amsmath}
\usepackage{amsfonts}
\usepackage{mathtools}
\mathtoolsset{showonlyrefs}
\usepackage{amssymb}

\usepackage{textcomp}
\usepackage{xcolor}

\usepackage{calrsfs}
\usepackage{csquotes}

\usepackage{tikz-cd}

\usepackage{comment}

\DeclareMathAlphabet{\mathcal}{OMS}{cmsy}{m}{n}

\theoremstyle{plain}
\newtheorem {theorem}{Theorem}[section]
\newtheorem {lemma}[theorem]{Lemma}
\newtheorem {corollary} [theorem]{Corollary}

\newtheorem {proposition} [theorem]{Proposition}
\theoremstyle{definition}

\newtheorem{example}[theorem]{Example}
\theoremstyle{remark}
\newtheorem{remark}[theorem]{Remark}

 \numberwithin{equation}{section} 
 
\newcommand{\R}{\mathbb{R}}

\newcommand{\N}{\mathbb{N}}

\newcommand{\D}{\mathcal{D}}

\newcommand{\oo}{\infty}

\newcommand{\e}{\varepsilon}
\newcommand{\s}{\sigma}

\newcommand{\w}{\omega}

\renewcommand{\d}{\partial}
\renewcommand{\a}{\alpha}
\renewcommand{\b}{\beta}
\newcommand{\g}{\gamma}

\newcommand{\de}{\delta}
\newcommand{\De}{\Delta}

\newcommand{\1}{{-1}}

\newcommand{\cut}{\mathrm{cut}}
\newcommand{\cor}{\mathrm{rect}}

\newcommand{\vspan}{\mathrm{span}}

\newcommand{\be}{\begin{equation}}
\newcommand{\ee}{\end{equation}}

\renewcommand{\k}{\mathrm k}

%% file: main.bbl
\begin{thebibliography}{10}

\bibitem{AS96}
A.~A. Agrachev and A.~V. Sarychev.
\newblock Abnormal sub-{R}iemannian geodesics: {M}orse index and rigidity.
\newblock {\em Ann. Inst. H. Poincar\'{e} Anal. Non Lin\'{e}aire},
  13(6):635--690, 1996.

\bibitem{ABB20}
Andrei Agrachev, Davide Barilari, and Ugo Boscain.
\newblock {\em A comprehensive introduction to sub-{R}iemannian geometry},
  volume 181 of {\em Cambridge Studies in Advanced Mathematics}.
\newblock Cambridge University Press, Cambridge, 2020.
\newblock With an appendix by Igor Zelenko.

\bibitem{CJMRSSS25}
Yacine Chitour, Frédéric Jean, Roberto Monti, Ludovic Rifford, Ludovic
  Sacchelli, Mario Sigalotti, and Alessandro Socionovo.
\newblock Not all sub-{R}iemannian minimizing geodesics are smooth.
\newblock {\em Preprint arXiv}, 2025.

\bibitem{libro_Enrico}
Enrico~Le Donne.
\newblock Metric lie groups. carnot-carath\'eodory spaces from the homogeneous
  viewpoint, 2024.

\bibitem{HL16}
Eero Hakavuori and Enrico Le~Donne.
\newblock Non-minimality of corners in subriemannian geometry.
\newblock {\em Invent. Math.}, 206(3):693--704, 2016.

\bibitem{LM08}
Gian~Paolo Leonardi and Roberto Monti.
\newblock End-point equations and regularity of sub-{R}iemannian geodesics.
\newblock {\em Geom. Funct. Anal.}, 18(2):552--582, 2008.

\bibitem{LS95}
Wensheng Liu and H\'{e}ctor~J. Sussman.
\newblock Shortest paths for sub-{R}iemannian metrics on rank-two
  distributions.
\newblock {\em Mem. Amer. Math. Soc.}, 118(564):x+104, 1995.

\bibitem{Mon94}
Richard Montgomery.
\newblock Abnormal minimizers.
\newblock {\em SIAM J. Control Optim.}, 32(6):1605--1620, 1994.

\bibitem{Mon02}
Richard Montgomery.
\newblock {\em A tour of subriemannian geometries, their geodesics and
  applications}, volume~91 of {\em Mathematical Surveys and Monographs}.
\newblock American Mathematical Society, Providence, RI, 2002.

\bibitem{S25}
Alessandro Socionovo.
\newblock Sharp regularity of sub-{R}iemannian length-minimizing curves.
\newblock {\em Preprint arXiv}, 2025.

\end{thebibliography}
